\documentclass{amsart}
\title{A Selection Principle and Products in Topological Groups}
\author{Marion Scheepers}
\date{March 20, 2022}

\newtheorem{theorem}{\bf Theorem}
\newtheorem{lemma}[theorem]{\bf Lemma}
\newtheorem{corollary}[theorem]{\bf Corollary}
\newtheorem{definition}[theorem]{\bf Definition}

\newtheorem{problem}[theorem]{\bf Problem}

\newcommand{\naturals}{\mathbb{N}}

\newcommand{\sone}{{\sf S}_1}
\newcommand{\gone}{{\sf G}_1}
\newcommand{\sfin}{{\sf S}_{fin}}
\newcommand{\gfin}{{\sf G}_{fin}}

\newcommand{\open}{\mathcal{O}}
\newcommand{\onbd}{\mathcal{O}_{nbd}}
\newcommand{\omeganbd}{\Omega_{nbd}}
\newcommand{\forces}{\mathrel{\|}\joinrel\mathrel{-}}
\newcommand{\poset}{{\mathbb P}}

\usepackage{xcolor}

\begin{document}

\maketitle
\begin{abstract}
 We consider the preservation under products, finite powers, and forcing, of a selection principle based covering property of $T_0$ topological groups. Though the paper is in part a survey, it contributes some new information, including:
 \begin{enumerate}
 \item{The product of a strictly o-bounded group with an o-bounded group is an o-bounded group - Corollary \ref{cor:stroboundedproducts}}
 \item{In the generic extension by a finite support iteration of $\aleph_1$ Hechler reals the product of any o-bounded group with a ground model $\aleph_0$ bounded group is an o-bounded group - Theorem \ref{thm:hechleriteration}}
 \item{In the generic extension by a countable support iteration of Mathias reals the product of any o-bounded group with a ground model $\aleph_0$ bounded group is an o-bounded group - Theorem \ref{thm:mathiasiteration}}
 \end{enumerate}
\end{abstract}
\section*{Introduction} 

In this paper we consider selection principles for open covers on topological space under the imposition of three major constraints: The topological spaces are assumed to be $\textsf{T}_0$, are assumed to be topological groups, and these topological groups are $\aleph_0$-bounded (a notion due to Guran, and defined below). 

Even under these three constraints there is a broad range of considerations regarding the relevant selection principles, and we shall also confine attention to a specific class of selection principles, and specific concerns regarding these. To give an initial indication of the scope of work considered here, recall: The following two selection principles, among several, are historically well-studied in several mathematical contexts: Let families $\mathcal{A}$ and $\mathcal{B}$ of sets be given. The symbol $\sfin(\mathcal{A},\mathcal{B})$ denotes the statement that there is for each sequence $(A_n:n\in\mathbb{N})$ of members of the family $\mathcal{A}$, a corresponding sequence $(B_n:n\in\mathbb{N})$ such that for each $n$, $B_n$ is a finite subset of $A_n$, and $\bigcup\{B_n:n\in\mathbb{N}\}$ is a set in the family $\mathcal{B}$. The symbol $\sone(\mathcal{A},\mathcal{B})$ denotes the statement that there is for each sequence $(A_n:n\in\mathbb{N})$ of members of the family $\mathcal{A}$, a corresponding sequence $(B_n:n\in\mathbb{N})$ such that for each $n$, $B_n$ is a member of $A_n$, and $\{B_n:n\in\mathbb{N}\}$ is a set in the family $\mathcal{B}$. It is well known that if $\mathcal{A} \subseteq \mathcal{B}$ and if $\mathcal{C}\subseteq \mathcal{D}$, then the following implications (more broadly illustrated in Figure \ref{fig:monotone}) hold: $\sfin(\mathcal{B},\mathcal{C}) \Rightarrow \sfin(\mathcal{A},\mathcal{D})$, $\sone(\mathcal{B},\mathcal{C}) \Rightarrow \sone(\mathcal{A},\mathcal{D})$, and $\sone(\mathcal{A},\mathcal{B})\Rightarrow \sfin(\mathcal{A},\mathcal{B})$.

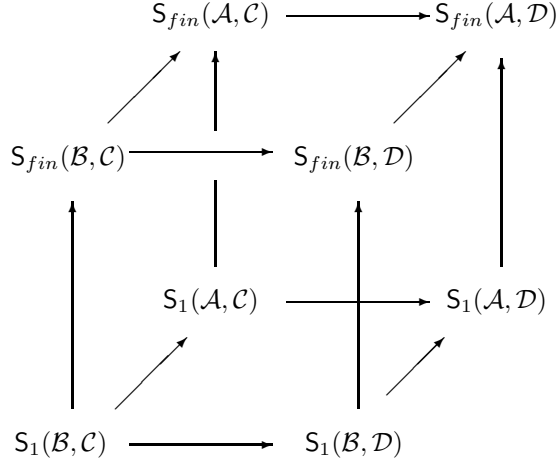
\begin{figure}[h]
\unitlength=.95mm
\begin{picture}(140.00,80.00)(10,10)
\put(45.00,20.00){\makebox(0,0)[cc]
{\shortstack {$\sone(\mathcal{B},\mathcal{C})$         } }}. 
\put(85.00,20.00){\makebox(0,0)[cc]
{\shortstack {$\sone(\mathcal{B},\mathcal{D})$   } }} 
\put(65.00,40.00){\makebox(0,0)[cc]
{\shortstack {$\sone(\mathcal{A},\mathcal{C})$   } }} 
\put(105.00,40.00){\makebox(0,0)[cc]
{\shortstack {$\sone(\mathcal{A},\mathcal{D})$   } }} 

\put(45.00,60.00){\makebox(0,0)[cc]
{\shortstack {$\sfin(\mathcal{B},\mathcal{C})$   } }} 
\put(85.00,60.00){\makebox(0,0)[cc]
{\shortstack {$\sfin(\mathcal{B},\mathcal{D})$   } }} 
\put(65.00,80.00){\makebox(0,0)[cc]
{\shortstack {$\sfin(\mathcal{A},\mathcal{C})$  } }} 
\put(105.00,80.00){\makebox(0,0)[cc]
{\shortstack {$\sfin(\mathcal{A},\mathcal{D})$  } }} 

\put(75.00,80.00){\vector(1,0){20.00}} 
\put(75.00,40.00){\vector(1,0){20.00}} 
\put(51.00,25.00){\vector(1,1){10.00}} 
\put(50.00,65.00){\vector(1,1){10.00}} 
\put(90.00,65.00){\vector(1,1){10.00}} 
\put(53.00,61.00){\vector(1,0){20.00}} 
\put(65.00,45.00){\line(0,1){12.00}} 
\put(65.00,64.00){\vector(0,1){11.00}} 
\put(45.00,25.00){\vector(0,1){29.00}} 
\put(85.00,25.00){\vector(0,1){29.00}} 
\put(89.00,27.00){\vector(1,1){8.00}} 
\put(105.00,45.00){\vector(0,1){29.00}} 
\put(53.00,20.00){\vector(1,0){20.00}} 
\end{picture}
\caption{Monotonicity Properties: $\mathcal{A}\subset\mathcal{B}$ and $\mathcal{C}\subset \mathcal{D}$.  \label{fig:monotone}}
\end{figure}

If instead of giving an entire antecedent sequence $(A_n:n\in\mathbb{N})$ of items from family $\mathcal{A}$ all at once for a selection principle and then producing a consequent sequence $(B_n:n\in\mathbb{N})$ to confirm that for example $\sfin(\mathcal{A},\mathcal{B})$ (or $\sone(\mathcal{A},\mathcal{B})$) holds, one can define a competition between two players, named ONE and TWO, where in inning $n$ ONE chooses an element $A_n$ from $\mathcal{A}$, and TWO responds with a $B_n$ from TWO's eligible choices. The players play an inning per positive integer $n$, producing a play
\begin{equation}\label{eq:play}
  A_1\; B_1\; A_2\; B_2\; \cdots\; A_m\; B_m\; \cdots
\end{equation}  
In the game named $\gfin(\mathcal{A},\mathcal{B})$ the play in (\ref{eq:play}) is won by TWO if for each $n$, $B_n$ is a finite subset of $A_n$ and $\bigcup\{B_n:n\in\mathbb{N}\}$ is an element of $\mathcal{B}$ - otherwise, ONE wins. In the game named $\gone(\mathcal{A},\mathcal{B})$ the play in (\ref{eq:play}) is won by TWO if for each $n$ $B_n\in A_n$, and $\{B_n:n\in\mathbb{N}\}$ is an element of $\mathcal{B}$. Observe that when ONE does not have a winning strategy in the game $\gfin(\mathcal{A},\mathcal{B})$, then $\sfin(\mathcal{A},\mathcal{B})$ is true. Similarly, when ONE does not have a winning strategy in the game $\gone(\mathcal{A},\mathcal{B})$, then $\sone(\mathcal{A},\mathcal{B})$ is true. The relationship between existence of winning strategies of a player and the corresponding properties of the associated selection principle is a fundamental question, and answers often reveal significant mathematical information. 

In this paper we consider the selection principle $\sfin(\mathcal{A},\mathcal{B})$ 
in the context where the families $\mathcal{A}$ and $\mathcal{B}$ are types of open covers arising in the study of topological groups. In the context of topological groups and the classes of open covers of these considered, there are some equivalences between the $\sfin(\cdot,\cdot)$ and $\sone(\cdot,\cdot)$ selection principles, as will be pointed out.

As noted initially, we throughout assume that the topological groups being considered have the $\textsf{T}_0$ separation property and thus, by the following classical theorem, the $\textsf{T}_{3\frac{1}{2}}$ separation property: 
\begin{theorem}[Kakutani, Pontryagin] \label{thm:T0} Any $T_0$ topological group is $T_{3\frac{1}{2}}$. 
\end{theorem}

In Section \ref{sec:aleph0bded} we briefly describe the resilience of $\aleph_0$-bounded groups under certain mathematical constructions and contrast these with the more constrained classical Lindel\"of property. In Section \ref{sec:products} we consider, for groups satisfying the targeted instance of the selection principle $\sfin(\mathcal{A},\mathcal{B})$,  
the preservation of the
selection property under the product construction. Though that there is a significant extant body of work on this topic, only some of these works and motivating mathematical questions relevant to the topic of Section \ref{sec:products} will be mentioned. In Section \ref{sec:cardinality} we make a brief excursion into exploring the cardinality of a class of groups emerging from product considerations in Section \ref{sec:products}. In Section \ref{sec:finpowers} we focus attention on groups for which finite powers satisfy the instance of $\sfin(\mathcal{A},\mathcal{B})$ being considered in this paper.  
In Section \ref{sec:Pgroups} we briefly return to a specific class of $\aleph_0$ bounded topological groups featured earlier in the paper. 

For background on topological groups we refer the reader to \cite{HR} and to \cite{Tkachenko}. For relevant background on forcing we refer the reader to \cite{Baumgartner}, \cite{Jech}, \cite{JechMF} and \cite{Kunen}.
Finally, as the reader will notice, this paper is part survey of known results, part investigation of refining or providing additional context for known results. The author would like to thank the editor of the volume for the flexibility in time to construct this paper.

\section{Open Covers and Fundamental Theorems}\label{sec:aleph0bded}

Besides the typical types of open covers considered for general topological spaces, there are also specific types of open covers considered in the context of topological groups. We introduce notation here for efficient reference to the various types of open covers relevant to this paper. Thus, let $(G,\otimes)$ denote a generic $\textsf{T}_0$ topological group, where $G$ is the set of elements of the group and $\otimes$ is the group operation. The symbol $id$ will denote the identity element of the group. It is also common practice to talk about the group $G$, without mentioning an explicit symbol for the operation. 

For an element $x$ of the group $G$ and for a nonempty subset $S$ of $G$, define
\begin{equation}\label{eq:pttimesset}
  x\otimes S = \{x\otimes g: g\in S\}.
\end{equation}
Observe that if $(G,\otimes)$ is a topological group, then when $S$ is an open subset of $G$, so is $x\otimes S$ for each element $x$ of $G$. Moreover, if $id$ is an element of $S$, then $x$ is an element of $x\otimes S$. Moreover, when $S$ and $T$ are nonempty subsets of $G$,
\begin{equation}\label{eq:settimesset}
 S\otimes T = \{x\otimes y:x\in S \mbox{ and }y\in T\}.
\end{equation}
Now we introduce notation for types of open covers of $(G,\otimes)$ to be considered here. 
\begin{itemize}
\item{$\open$: The set of all open covers of $G$.}
\item{$\onbd$: The set of all open covers of $G$ of the following form: For a neighborhood $U$ of $id$, $\open(U)$ denotes the open cover $\{x\otimes U:x\in G\}$ of $G$, and $\onbd$ denotes the collection $\{\open(U): U \mbox{ a neighborhood of } id\}$.}
\item{$\Omega$: An open cover $\mathcal{U}$ of $G$ is said to be an $\omega$-cover (originally defined in \cite{GN}) if $G$ itself is not a member of $\mathcal{U}$, and for each finite subset $F$ of $G$ there is a $U\in\mathcal{U}$ such that $F\subseteq U$. The symbol $\Omega$ denotes the set $\{\mathcal{U}:\; \mathcal{U} \mbox{ an }\omega \mbox{ cover of }G\}$}
\item{$\omeganbd$: The set of all open covers of $G$ of the following form: For a neighborhood $U$ of $id$, $\Omega(U)$ denotes the open cover $\{F\otimes U: F\subset G \mbox{ a finite set}\}$. The symbol $\omeganbd$ denotes the set of all open covers of the form $\Omega(U)$ of $G$.}
\item{$\Gamma$: An open cover $\mathcal{U}$ is a $\gamma$-cover (also introduced in \cite{GN}) if it is infinite and  for each $x\in G$, $x$ is a member for all but finitely sets in $\mathcal{U}$.  The symbol $\Gamma$ denotes the set $\{\mathcal{U}:\; \mathcal{U} \mbox{ a }\gamma \mbox{ cover of }G\}$}
\item{$\Lambda$: An open cover $\mathcal{U}$ is a large cover if for each $x\in G$, $x$ is a member of infinitely sets in $\mathcal{U}$. $\Lambda$ denotes the collection of large covers of $G$.}
\end{itemize}

Targeted properties related to topological objects, such as for example the preservation of a property of factor spaces in product spaces,  have led to the identification of several additional types of open covers for topological spaces. Some of these used in this paper are as follows:
\begin{itemize}
\item{$\mathcal{O}^{gp}$: An open cover $\mathcal{U}$ is an element of $\open^{gp}$ if it is infinite, and there is a partition $\mathcal{U} = \bigcup\{\mathcal{U}_n:n\in\mathbb{N}\}$ where for each $n$ the set $\mathcal{U}_n$ is finite, for all $m\neq n$, we have $\mathcal{U}_m\cap\mathcal{U}_n = \emptyset$, and each element of the underlying space is in each but finitely many of the sets $\bigcup\mathcal{U}_n$. We say that $\mathcal{U}$ is a \emph{groupable} cover.} 
\item{$\mathcal{O}^{wgp}$: An open cover $\mathcal{U}$ is an element of $\open^{wgp}$ if it is infinite, and there is a partition $\mathcal{U} = \bigcup\{\mathcal{U}_n:n\in\mathbb{N}\}$ where for each $n$ the set $\mathcal{U}_n$ is finite, for all $m\neq n$, we have $\mathcal{U}_m\cap\mathcal{U}_n = \emptyset$, and for each finite subset $F$ of the underlying space there is an $n$ such that $F\subseteq \bigcup\mathcal{U}_n$. We say that $\mathcal{U}$ is a \emph{weakly groupable} cover.}
\end{itemize}
From the definitions it is evident that the following inclusions hold among these types of open covers: 
$\Gamma\subset \open^{gp}\subset \open^{wgp} \subset \Lambda\subset \open$, 
 $\Gamma \subset \Omega\subset \open^{wgp}\subset \open$,  
$\omeganbd\subset \Omega$, $\omeganbd\subset \onbd$ and $\onbd\subset \open$.

For several traditional covering properties of topological spaces, natural counterparts are defined in the domain of topological groups by restricting the types of open covers considered in defining the covering properties. For example:
\begin{definition} A topological group is said to be 
\begin{enumerate}
\item{$\aleph_0$-bounded if it has the Lindel\"of property with respect to the family $\onbd$ of open covers: That is, each member of $\onbd$ has a countable subset that covers the group.}
\item{totally bounded (or pre-compact) if it is compact with respect to the family $\onbd$ of open covers: That is, each element of $\onbd$ has a finite subset covering the group.}
\item{$\sigma$-bounded if it is a union of countably many totally bounded subsets.}
\end{enumerate}
\end{definition}

Many of the properties of $\aleph_0$-bounded groups can be obtained from the following fundamental result:
\begin{theorem}[Guran]\label{thm:Guran1} A topological group is $\aleph_0$-bounded if, and only if, it embeds as a topological group into the Tychonoff product of second countable groups.
\end{theorem}
The $\aleph_0$-boundedness property is resilient under several mathematical constructions. For example, any subgroup of an $\aleph_0$-bounded group is $\aleph_0$-bounded. Also, the Tychonoff product of any number of $\aleph_0$-bounded groups is an $\aleph_0$-bounded group. These two facts in particular imply:
\begin{lemma}\label{lemma:cardinality} There is for each infinite cardinal number $\kappa$ an $\aleph_0$-bounded group of cardinality $\kappa$.
\end{lemma}

The $\aleph_0$-boundedness property and the total boundedness property are also resilient under forcing extensions of the set theoretic universe:
\begin{theorem}\label{boundedpreserve} If $(G,\otimes)$ is an $\aleph_0$-bounded (totally bounded) topological group and $({\mathbb P},<)$ is a forcing notion, then 
\[
  {\mathbf 1}_{\mathbb P}\forces``(\check{G},\otimes) \mbox{ is $\aleph_0$-bounded (respectively totally bounded)}".
\]
\end{theorem}
\begin{proof} Let $\dot{U}$ be a ${\mathbb P}$-name such that ${\mathbf 1}_{\mathbb P}\forces``\dot{U} \mbox{ is a neighborhood of the identity}"$. Choose a maximal antchain $A$ for ${\mathbb P}$ and for each $q\in A$ a neighborhood $U_q$ of the identity such that $q\forces``\dot{U} = \check{U}_q"$. We give an argument for $\aleph_0$-boundedness. The argument for totally bounded is similar.

Since $(G,\otimes)$ is $\aleph_0$-bounded, choose for each $q\in A$ a countable set $X_q:=\{x^q_n:n<\omega\}$ of elements of $G$ such that $X_q\otimes U_q=G$. Define $\dot{X} = \{(\check{x}^q_n,q):n<\omega \mbox{ and }q\in A\}$. Then $\dot{X}$ is a ${\mathbb P}$-name and 
\[
  {\mathbf 1}_{\mathbb P}\forces``\dot{X}\subseteq\check{G} \mbox{ is countable and }\dot{X}\otimes\dot{U} = \check{G}"
\]
Thus, ${\mathbf 1}_{\mathbb P}\forces``(\check{G},\otimes) \mbox{ is $\aleph_0$-bounded}"$. %
\end{proof} 

Proper forcing posets also preserves the property of not being $\aleph_0$-bounded:
\begin{theorem}\label{nonboundedpreserve}
Let $(G,\otimes)$ be a topological group which is not $\aleph_0$ bounded. Let $({\mathbb P},<)$ be a proper partially ordered set. Then
\[
  {\mathbf 1}_{\mathbb P}\forces``(\check{G},\otimes) \mbox{ is not $\aleph_0$-bounded}".
\]
\end{theorem}
{\flushleft{\bf Proof:}} For let $U$ be a neighborhood of the identity witnessing that $(G,\otimes)$ is not $\aleph_0$-bounded. Suppose that $p\in{\mathbb P}$ and ${\mathbb P}$-name $\dot{X}$ are such that $p\forces``\dot{X}\otimes \check{U} = \check{G} \mbox{ and }\dot{X}\subseteq\check{G} \mbox{ is countable}"$. Since ${\mathbb P}$ is a proper poset there is a countable set $C\subseteq G$ such that $p\forces \dot{X}\subseteq \check{C}"$ - \cite{JechMF}, Proposition 4.1. But then $p\forces``\check{C}\otimes \check{U}=\check{G}"$. Since all the parameters in the sentence forced by $p$ are in the ground model, we find the contradiction that $G = C\otimes U$.
$\Box$

Thus, when forcing with a proper forcing notion, a ground model topological group is $\aleph_0$-bounded in the generic extension if, and only if, it is $\aleph_0$-bounded in the ground model. Note, incidentally, that the same argument shows
\begin{theorem}\label{nonlindelofpreserve}
Let $(X,\tau)$ be a topological space which is not Lindel\"of. Let $({\mathbb P},<)$ be a proper partially ordered set. Then
\[
  {\mathbf 1}_{\mathbb P}\forces``(\check{X},\tau) \mbox{ is not Lindel\"of}".
\]
\end{theorem}

When considering a strengthening of the $\aleph_0$-boundedness property, resilience of the stronger property under a corresponding mathematical constructions is more subtle. For example, the Lindel\"of property requires that for any open cover (not only ones from $\onbd$) there is a countable subset that still is a cover. Every Lindel\"of group is an $\aleph_0$-bounded group, but not conversely. The Lindel\"of property is not in general preserved by subspaces, products, or forcing extensions.  Similarly, for subclasses (determined by selection principles) of the family of $\aleph_0$-bounded groups, preservation of membership to the subclass under Tychonoff products and behavior under forcing is more subtle. Also questions regarding cardinality of members of the more restricted family are more delicate. 

\section{Products and Groups with the Property $\sfin(\onbd,\open)$}\label{sec:products}

In the notation established here, a topological group is said to be \emph{o - bounded} if it has the property $\sone(\omeganbd,\open)$. In the literature the notion of an o-bounded group is attributed to Okunev. In Theorem 3 of \cite{BKS} it is proven that for a topological group the three properties $\sone(\omeganbd,\open)$, $\sfin(\onbd,\open)$ and $\sfin(\omeganbd,\open)$ are equivalent. The property $\sfin(\onbd,\open)$ is also known as Menger boundedness.

In Problem 5.2 of \cite{Hernandez} Hernandez asked:
\begin{problem}\label{problem:obdproducts}
 Is the product of two topological groups, each satisfying the property $\sfin(\onbd,\open)$, a topological group satisfying the property $\sfin(\onbd,\open)$?
\end{problem} 

Subsequently it was discovered - see Example 2.12 of \cite{HRT} - that there are groups $G$ and $H$, each satisfying the property $\sfin(\onbd,\open)$, for which the group $G\times H$ does not satisfy the property $\sfin(\onbd,\open)$. Since subgroups of a group satisfying $\sfin(\onbd,\open)$ inherit the property $\sfin(\onbd,\open)$, for $G\times H$ to have the property $\sfin(\onbd,\open)$, each of the groups $G$ and $H$ must have at least the property $\sfin(\onbd,\open)$. Thus Example 2.12 of \cite{HRT} demonstrates that $G$ or $H$ should satisfy additional hypotheses to guarantee that the product has the property $\sfin(\onbd,\open)$. Under what conditions on $G$ and $H$ would the product group $G\times H$ satisfy the property $\sfin(\onbd,\open)$? 
A number of \emph{ad hoc} additional conditions on a topological group $G$ that guarantee that its product with a group $H$ also has the property $\sfin(\onbd,\open)$ have been discovered. Here are two examples of such conditions:

\begin{theorem}[\cite{Hernandez} Theorem 5.3]\label{thm:sigmacpt} If $G$ is a subgroup of a $\sigma$-compact
 topological group and $H$ is an $\sfin(\onbd,\open)$ group, then $G\times H$ satisfies $\sfin(\onbd,\open)$. 
\end{theorem}

For the next example recall that a topological group is a $P$ group if, and only if, the intersection of countably many open neighborhoods of the identity element still is an open neighborhood of the identity element. More generally a topological space is a $P$ space if each countable intersection of open sets is an open set.

\begin{theorem}[\cite{HRT}, Theorem 2.4]\label{thm:Pgp} If $G$ is an $\aleph_0$-bounded $P$ group and the group $H$ satisfies the property $\sfin(\onbd,\open)$, then $G\times H$ satisfies $\sfin(\onbd,\open)$.
\end{theorem}

Though the conditions in Theorems \ref{thm:sigmacpt} and \ref{thm:Pgp} at first glance seem very different, a single unifying property in the literature implies both results, namely:
\begin{theorem}[\cite{LB1}, Theorem 6]\label{thm:productgroups}
Let $(G,\otimes)$ be a $T_0$ topological group satisfying the selection principle $\sone(\omeganbd,\Gamma)$. Let $\mathcal{A}$ be any of $\open$, $\Omega$ or $\Gamma$. If $(H,\triangle)$ is a topological group satisfying $\sone(\omeganbd,\mathcal{A})$, then the product group $G\times H$ satisfies $\sone(\omeganbd,\mathcal{A})$.
\end{theorem}

To obtain Theorem \ref{thm:sigmacpt} from Theorem \ref{thm:productgroups}, observe
\begin{lemma}\label{lemma:Th5.3}
An infinite $\sigma$-compact group, as well as any infinite subgroup of it, has the property $\sone(\omeganbd,\Gamma)$.
\end{lemma}
\begin{proof} We give the argument for infinite $\sigma$-compact groups, leaving the proof for subgroups of such groups to the reader.
Assume that $G$ is $\sigma$-compact, and write $G$ as the union $\bigcup\{G_n:n\in\mathbb{N}\}$, where for each $n$ $G_n$ is compact and $G_n\subset G_{n+1}$. Let $(\mathcal{U}_n:n\in\mathbb{N})$ be a sequence of $\omeganbd$ covers of $G$. For each $n$ fix a neighborhood $U_n$ of the identity element such that $\mathcal{U}_n = \{F\otimes U_n: F\subset G \mbox{ finite}\}$.

For each $n$, as $G_n$ is compact, choose a finite set $F_n\subset G$ such that $G_n\subseteq S_n = F_n\otimes U_n \in \mathcal{U}_n$. Then the sequence $(S_n:n\in\mathbb{N})$ witnesses $\sone(\omeganbd,\Gamma)$ for the given sequence $(\mathcal{U}_n:n\in\mathbb{N})$ of $\omeganbd$ covers of $G$.
\end{proof}
Next we show how to derive Theorem \ref{thm:Pgp} from Theorem \ref{thm:productgroups}. First, using the argument in Theorem 2.4 of \cite{Hernandez},
\begin{lemma}\label{lemma:obPgp1}
If $(G,\otimes)$ is an $\aleph_0$-bounded $P$ group, then it has the property $\sone(\onbd,\open)$
\end{lemma}
\begin{proof}
Let a sequence $(\mathcal{U}_n:n\in\mathbb{N})$ of $\onbd$-covers of $G$ be given. For each $n$ choose a neighborhood $M_n$ of the identity element such that $\mathcal{U}_n = \mathcal{O}(M_n)$. Since $G$ is a $P$ group, $M = \bigcap\{M_n:n\in\mathbb{N}\}$ is an open set, and neighborhood of the identity element. Then $\mathcal{U} = \{x\otimes M:x\in G\}$ is a member of $\onbd$. Since $G$ is $\aleph_0$-bounded, fix a countable set $\{x_n:n\in \mathbb{N}\}$ of elements of $G$ such that $\{x_n\otimes M:n\in\mathbb{N}\}$ is a cover of $G$. Then for each $n$ also $x_n\otimes M \subseteq x_n\otimes M_n$. Thus $\{x_n\otimes M_n:n\in\mathbb{N}\}$ witnesses for the sequence $(\mathcal{U}_n:n\in\mathbb{N}\}$ that $(G,\otimes)$ has the property $\sone(\onbd,\open)$.
\end{proof}

Since $\sone(\onbd,\open)$ implies $\sfin(\onbd,\open)$, the following lemma extends the conclusion of Lemma \ref{lemma:obPgp1}:
\begin{lemma}\label{lemma:obPgp2}
If $(G,\otimes)$ is an $\aleph_0$-bounded $P$ group, then it has the property $\sone(\omeganbd,\Omega)$
\end{lemma}
\begin{proof}
Finite products of $P$ spaces are P spaces. But any (Tychonoff) product of $\aleph_0$-bounded groups is $\aleph_0$-bounded (see for example Proposition 3.2 in the survey \cite{Tkachenko}). Thus, any finite product of $\aleph_0$-bounded $P$ groups is an $\aleph_0$ bounded P-group. By \cite{HRT} Theorem 2.4, finite products of $\aleph_0$-bounded P groups are $\sfin(\onbd,\open)$. By Theorems 2 and 4 of \cite{BKS}, $G$ satisfies $\sone(\omeganbd,\Omega)$.
\end{proof}

And finally, we strengthen the conclusion of Lemma \ref{lemma:obPgp2}. 
\begin{theorem}\label{thm:PgpSel}\footnote{An alternative proof of Theorem \ref{thm:PgpSel} is given below by Lemmas \ref{lemma:obPgp3} and \ref{lemma:gptogamma}.}
Any $\aleph_0$-bounded $P$ group has the property $\sone(\omeganbd,\Gamma)$
\end{theorem}
\begin{proof} Let $(G,\otimes)$ be an $\aleph_0$-bounded $P$ group.  
Let $(V_n:n\in\mathbb{N}$ be a sequence of neighborhoods of the identity element of $G$. For each $n$ choose a finite set $F_n$ such that $\{F_n\otimes V_n:n\in\mathbb{N}\}$ is an $\omega$-cover of $G$. Put $V = \bigcap\{U_n:n\in\mathbb{N}\}$. Since $G$ is a $P$-space, $V$ is an open neighborhood of the identity. Also for each $n$ $\mathcal{U} = \{F\otimes V:F\subset G \mbox{ finite}\}$ is an $\omeganbd$ cover refining $\mathcal{U}_n = \{F\otimes V_n: F\subset G \mbox{ finite}\}$. Applying $\sone(\omeganbd,\Omega)$ to the sequence $(\mathcal{U},\mathcal{U},\mathcal{U},\cdots)$ fix for each $n$ a finite set $S_n\subset G$ such that $\{S_n\otimes V:n\in\mathbb{N}\}$ is an $\omega$-cover. For each $n$, set $G_n = \bigcup\{S_j:j\le n\}$. Then for each $n$, $G_n\otimes V_n\in \mathcal{U}_n$, and $\{G_n\otimes V_n:n\in\mathbb{N}\}$ is a $\gamma$-cover of $G$.
\end{proof}
{\flushleft Finally, Theorem \ref{thm:productgroups} and Theorem \ref{thm:PgpSel} imply Theorem \ref{thm:Pgp}.}

Continuing with the theme of providing a single unifying property for questions and claims regarding preserving the property $\sone(\omeganbd,\open)$ in products, we also give a result on a question from the literature. Tkachenko defined a topological group to be \emph{strictly o-bounded} if player TWO has a winning strategy in the game $\gone(\omeganbd,\open)$\footnote{Equivalently, TWO has a winning strategy in the game $\gfin(\onbd,\open)$} - \cite{Hernandez}. 
In Problem 2.4 of \cite{HRT} the authors ask
\begin{problem}\label{problem:HRT} Is it true that whenever $(G,\otimes)$ is a strictly o-bounded group and $(H,\triangle)$ satisfies the property $\sfin(\onbd,\open)$, then $G\times H$ also satisfies the property $\sfin(\onbd,\open)$? 
\end{problem}
Problem \ref{problem:HRT} was partially answered in Corollary 8 of \cite{LB1} for the case when the strictly o-bounded group  $(G,\otimes)$ is metrizable. Towards answering Problem \ref{problem:HRT} we generalize a part of Theorem 5 of \cite{LB1}.

\begin{theorem}\label{thm:sobishb} If player TWO has a winning strategy in the game $\gone(\omeganbd,\open)$ played on a $\textsf{T}_0$ topological group, then that group has the property $\sone(\omeganbd,\Gamma)$. 
\end{theorem}
\begin{proof} Let $(G,\otimes)$ be a strictly o-bounded group. By Lemma \ref{lemma:Th5.3} we may assume it is not totally bounded. Assume that TWO has a winning strategy in the game, say it is $\sigma$. Let $(U_n:n\in\naturals)$ be a sequence of neighborhoods of the identity, each witnessing that the group is not totally bounded. For each $n$ let $\Omega(U_n)$ denote $\{F\otimes U_n:F\subset G\mbox{ finite}\}$, an element of $\omeganbd$ for $G$..

Then $(\Omega(U_n):n\in\naturals)$ is a sequence of elements of $\omeganbd$. In the game $\gone(\omeganbd,\open)$ ONE chooses elements of $\omeganbd$ and TWO selects members of ONE's moves. Following the construction in the proof of 1$\Rightarrow 2$ of Theorem 5 of \cite{LB1}, define the following subsets of $G$:
\begin{equation}\label{firstmove}
G_{\emptyset} = \bigcap_{n\in\naturals}\sigma(\Omega(U_n)).
\end{equation}
For $\tau = (n_1,\cdots,n_k)$ a finite sequence of positive integers, define
\begin{equation}\label{latermove}
G_{\tau} = \bigcap_{n\in\naturals}\sigma(\Omega(U_{n_1}),\, \cdots,\, \Omega(U_{n_k}),\, \Omega(U_n)).
\end{equation}

{\flushleft{\bf Claim 1: }} $G = \bigcup_{\tau\in\,^{<\omega}\naturals} G_{\tau}$\\ 
 For suppose on the contrary that $x\in G$ is not an element of the union $\bigcup_{\tau\in\,^{<\omega}\naturals} G_{\tau}$. As $x$ is not in $G_{\emptyset}$, choose $n_1$ with $x\not\in \sigma(\Omega(U_{n_1}))$. Then as $x$ is not in $G_{n_1}$ choose an $n_2$ with $x\not\in \sigma(\Omega(U_{n_1}),\Omega(U_{n_2}))$, and so on. In this way we find a $\sigma$-play of the game during which TWO never covered $x$, contradicting the hypothesis that $\sigma$ is a winning strategy for TWO. 

{\flushleft{\bf Claim 2:}} For each finite sequence $\tau$ of positive integers, and for each $n$ there is a finite set $F\subseteq G$ such that $G_{\tau}\subseteq F*U_n$. \\
For let $\tau=(n_1,\cdots,n_k)$ as well as $n$ be given. Then 
\[
 G_{\tau}\subseteq \sigma(\Omega(U_{n_1}),\cdots,\Omega(U_{n_k}),\Omega(U_n))\in\Omega(U_n).
\]
Finally, enumerate the set of finite sequences of positive integers as $\tau_1,\, \tau_2,\, \cdots,\, \tau_n,\,\cdots$. Choose finite subsets $F_1,\, F_2,\, \cdots,\, F_n,\,\cdots$ of $G$ so that for each $k$ we have
\[
  G_{\tau_1}\cup\cdots\cup G_{\tau_k}\subseteq F_k\otimes U_k \in \Omega(U_k).
\]
The sequence $(F_k*U_k:k\in\naturals)$ witnesses $\sone(\omeganbd,\Gamma)$ for the given sequence of neighborhoods of the identity.
\end{proof}

The following corollary answers Problem \ref{problem:HRT}:
\begin{corollary}\label{cor:stroboundedproducts}
If $(G,\otimes)$ is a strictly o-bounded $\textsf{T}_0$ group, and $(H,\triangle)$ is a $\textsf{T}_0$ group with the property $\sfin(\onbd,\open)$, then $G\times H$ has the property $\sfin(\onbd,\open)$.
\end{corollary}
\begin{proof}
Let $(G,\otimes)$ and $(H,\triangle)$ be as in the hypotheses. By Theorem \ref{thm:sobishb} the group $(G,\otimes)$ has the property $\sone(\omeganbd,\Gamma)$. Then by Theorem \ref{thm:productgroups}, $G\times H$ has the property $\sone(\omeganbd,\open)$.
\end{proof}

\section{The cardinality of $T_0$ groups with the property $\sone(\omeganbd,\Gamma)$.} \label{sec:cardinality}

Next we briefly consider the cardinality of topological groups satisfying the property $\sone(\omeganbd,\Gamma)$. It is useful to first catalogue a few basic behaviors of the property $\sone(\omeganbd,\Gamma)$ under some standard forcing notions. Although one can prove that in general any forcing iteration of length of uncountable cofinality which cofinally often adds a dominating real converts any ground model $\aleph_0$-bounded group into a group satisfying $\sone(\omeganbd,\Gamma)$, we prove it here for a specific partially ordered set: 

\begin{theorem}\label{thm:hechleriteration} Let $\kappa$ be a cardinal number of uncountable cofinality. Let $(\poset, <)$ be the finite support iteration by $\kappa$ Hechler reals. If $(G,\otimes)$ is $\aleph_0$-bounded in the ground model, then 
\[
  {\mathbf 1}_{\poset}\forces ``\check{G} \mbox{ has the property }\sone(\omeganbd,\Gamma)".
\]  
\end{theorem}
{\flushleft{\bf Proof:}} By Theorem \ref{boundedpreserve} ${\mathbf 1}_{\poset}\forces ``\check{(G,\otimes)} \mbox{ is $\aleph_0$ bounded}"$. Thus, as $\poset$ has the countable chain condition, if we take a $\poset$-name $(\dot{\mathcal{U}}_n:n<\omega)$ for a sequence of $\onbd$ members we may assume that this sequence is present in the ground model, since it is a name in an initial segment of the iteration, and we can factor the iteration at this initial segment. Since in this initial segment $(G,\otimes)$ is $\aleph_0$-bounded we may choose for each $n$ a countable subset $X_n$ of $G$ such that $G = \bigcup_{n<\omega}X_n\otimes U_n$, where $\mathcal{U}_n = \open(U_n)$. Define for each $x\in G$ a function $f_x$ from $\omega$ to $\omega$ as follows: Enumerate $X_n$ as $(x^n_m:m<\omega)$. Then 
\[
  f_x(n) = \min\{m:x\in \{x^n_1,\cdots,x^n_m\}\otimes U_n\}.
\]
The family $\{f_x:x\in G\}$ is in an initial segment of the iteration,  
and so the next Hechler real added eventually dominates each $f_x$. Let $g$ be the next Hechler real. Then $\{x^n_j:j\le g(n)\}\otimes U_n \subseteq \mathcal{U}_n$ is a finite subset of $\mathcal{U}_n$, and for each $x$, for all but finitely many $n$, $x\in \{x^n_j:j\le g(n)\} \otimes U_n$. It follows that the group $(G,\otimes)$ has the property $\sone(\omeganbd,\Gamma)$.
$\Box$

Incidentally, the Hechler reals partially ordered set does not preserve the Lindel\"of property. In Remark 5 of \cite{IG} Gorelic points out that the points ${\sf G}_{\delta}$ Lindel\"of subspace in this model fails to be Lindel\"of in the generic extension that forces MA plus not-CH. Indeed, this can be accomplished by a finite support iteration of $\omega_2$ or more Hechler reals over a model of CH. Readers could consult the original paper by Hechler \cite{Hechler}, or for example \cite{Palumbo} on Hechler real generic extensions 

\begin{theorem}\label{thm:mathiasiteration} Let $(\poset, <)$ be the countable support iteration by $\aleph_2$ Mathias reals over a model of CH. If $(G,\otimes)$ is $\aleph_0$-bounded in the ground model, then 
\[
   {\mathbf 1}_{\poset}\forces ``\check{(G,\otimes)} \mbox{ has the property } \sone(\omeganbd,\Gamma)".
\]   
\end{theorem}
\begin{proof} By Theorem \ref{boundedpreserve} ${\mathbf 1}_{\poset}\forces ``\check{G} \mbox{ is $\aleph_0$ bounded}"$. Thus, as CH holds and antichains of the poset $(\mathbb{P},<)$ have cardinality at most $\aleph_1$, for any $\poset$-name $(\dot{U}_n:n<\omega)$ for a sequence of neighborhoods of the identity, we may assume that this sequence of neighborhoods of the identity is present in the ground model (the name of the sequence is a name in an initial segment of the iteration), and factor the iteration over this initial segment. Since by Theorem \ref{boundedpreserve} $(G,\otimes)$ is $\aleph_0$-bounded in this initial segment choose (in the generic extension by this initial segment) for each $n$ a countable subset $X_n$ of $G$ such that $G = \bigcup_{n<\omega}X_n\otimes U_n$, where $\mathcal{U}_n = \open(U_n)$. Define for each $x\in G$ a function $f_x$ from $\omega$ to $\omega$ as follows: Enumerate $X_n$ as $(x^n_m:m<\omega)$. Then 
\[
  f_x(n) = \min\{m:x\in \{x^n_1,\cdots,x^n_m\}\otimes U_n\}.
\]
The family $\{f_x:x\in G\}$ is in the generic extension by the initial segment (the ``ground model" fooro the remaining generic extension), and so the next Mathias real added by the generic extension eventually dominates each $f_x$. Let $g$ be such a dominating real. Then $\{x^n_j:j\le g(n)\}\otimes U_n \subseteq \mathcal{U}_n$ is a finite subset of $\mathcal{U}_n$, and for each $x$, for all but finitely many $n$, $x\in \{x^n_j:j\le g(n)\} \otimes U_n$. It follows that in the generic extension $(G,\otimes)$ has the property $\sone(\omeganbd,\Gamma)$.
\end{proof}

As a consequence we obtain 
\begin{theorem}\label{thm:Hurewiczcards}
It is consistent, relative to the consistency of \textsf{ZFC}, that there is for each cardinal number $\kappa$ a group with property $ \sone(\omeganbd,\Gamma)$. 
\end{theorem}
\begin{proof}
By Lemma \ref{lemma:cardinality} there exists for each infinite cardinal number $\kappa$ an $\aleph_0$-bounded group of cardinality $\kappa$. By either of Theorem \ref{thm:hechleriteration} or Theorem \ref{thm:mathiasiteration}, in the corresponding generic extension each ground model $\aleph_0$-bounded group has property $\sone(\omeganbd,\Gamma)$. Since the forcing partially ordered set in either case preserves cardinal numbers, the result follows.
\end{proof}

To round off the consideration of the property $\sone(\omeganbd,\Gamma)$ under forcing it is worth recording for the record that 
\begin{theorem}\label{thm:Hurewiczcccpreserve} If the group  $(G,\otimes)$ has the property $\sone(\omeganbd,\Gamma)$ and if $(\poset,<)$ is a partially ordered set with the countable chain condition,, then 
\[
  {\mathbf 1}_{\poset}\forces ``(\check{G},\check{\otimes}) \mbox{ has the property }\sone(\omeganbd,\Gamma)".
\]  
\end{theorem}
\begin{proof}
Let $(\dot{U}_n:n<\omega)$ be a $\poset$-name for a sequence of neighborhoods of the identity element of the group $(G,\otimes)$.
For each $n$ choose (in the ground model) a sequence $(U^n_m:m<\omega)$ of neighborhoods of the identity element of $(G,\otimes)$, and a maximal antichain $(q^n_m:m<\omega)$ of $\poset$ such that for each $n$ and $m$
$   q^n_m\forces ``\check{U}^n_m\subseteq \dot{U}_n".
$  
 Then 
$  \dot{V}_n = \{(\check{U}^n_m,q^n_m):m<\omega\}
$ 
is a $\poset$-name and
\[
  {\mathbf 1}_{\poset}\forces``\dot{V}_n\subseteq\dot{U}_n \mbox{ is a neighborhood of the identity element of } \check{G}"
\]

For each $n$ define
$   N_n = \bigcap_{k,\ell\le n}U^k_{\ell},
$ 
a (ground model) neighborhood of the identity element of $(G,\otimes)$.
Applying the property $\sone(\omeganbd,\Gamma)$, choose finite sets $F_1\subseteq F_2\subseteq\cdots$ such that for each $x\in G$, for all but finitely many $k$, $x$ is a member of $F_k\otimes N_k$. 
Then for each $n$ define the $\poset$-name $\dot{F}_n$ fooor a finite subset of $\check{G}$ by $\{(\check{F}_{n+m},q^n_m):m<\omega\}$.

{\flushleft{\bf Claim:}} ${\mathbf 1}_{\poset}\forces``(\forall x\in\check{G})(\forall^{\infty}_n)(x\in \dot{F}_n\otimes\dot{V}_n)"$.\\

For let $H$ be a $\poset$-generic filter. For each $n$ choose $m_n$ with $q^n_{m_n}\in H$. Then we have that for each $n$,
$   (\dot{F}_n)_{H} = F_{n+m_n}.
$]  
Consider any $x\in G$. Choose $k$ s0 large that for $n\ge k$ we have $n+m_n>k$ and $x\in F_{n+m_n}\otimes N_{n+m_n}$. Since
$   N_{n+m_n}\subseteq U^n_{m_n} = (\dot{V})_H
$ 
it follows that $x\in (\dot{F}_n\otimes\dot{V}_n)_H$.   
\end{proof}

\section{Finite powers of Groups with the property $\sfin(\onbd,\open)$}\label{sec:finpowers}

Consider a topological group $(G,\otimes)$ has the property that whenever $(H,\triangle)$ is a topological group with the property $\sfin(\onbd,\open)$, then the product group $G\times H$ also has the property $\sfin(\onbd,\open)$. Then necessarily the group $G\times G$ has the property $\sfin(\onbd,\open)$: Indeed, every finite power of the group $(G,\otimes)$ has the property $\sfin(\onbd,\open)$.

Recall Example 2.12 of \cite{HRT} which illustrates that the product of two groups, each with the property $\sfin(\onbd,\open)$, does not necessarily have the property $\sfin(\onbd,\open)$. This example in fact gives a group $(G,\otimes)$ which has property $\sfin(\onbd,\open)$ but $(G,\otimes)\times (G,\otimes)$ does not have property $\sfin(\onbd,\open)$ (and the group is even metrizable).  
 One might ask whether the phenomenon exhibited by this example - $(G,\otimes)$ has property $\sfin(\onbd,\open)$, but $(G,\otimes)\times(G,\otimes)$ does not  -  is the only obstruction to a topological group $(G,\otimes)$ having a property such as 
\begin{itemize}
\item[(A)]{the product of $(G,\otimes)$ with any group with property $\sfin(\onbd,\open)$ has the property $\sfin(\onbd,\open)$}
\item[(B)]{each finite power of $(G,\otimes)$ has property $\sfin(\onbd,\open)$.}
\end{itemize} 
The following two prior results shed significant light on version (B) of this question:
\begin{theorem}[Banakh and Zdomskyy, Mildenberger and Shelah]\label{thm:finpowersprod} The following statement is consistent, relative to the consistency of \textsf{ZFC}:
{\flushleft{For}} each $\textsf{T}_0$ group $(G,\otimes)$, if $(G,\otimes)\times(G,\otimes)$ has the property $\sfin(\onbd,\open)$, then the group $(G,\otimes)$ in fact has the property $\sfin(\omeganbd,\open^{wgp})$.
\end{theorem}

Regarding Theorem \ref{thm:finpowersprod}: Prior results (Theorems 3, 6 and 7 of \cite{BKS}) that show that every finite power of a topological group has property $\sfin(\onbd,\open)$ if, and only if, the group has the property $\sone(\omeganbd,\open^{wgp})$. Moreover,
\begin{lemma}\label{lemma:sfinsone} For a topological group $(G,\otimes)$ the following are equivalent:
\begin{enumerate}
\item{$(G,\otimes)$ has the property $\sfin(\omeganbd,\open^{wgp})$}
\item{$(G,\otimes)$ has the property $\sone(\omeganbd,\open^{wgp})$}
\end{enumerate}
\end{lemma}
\begin{proof}
We must show that $(1)$ implies $(2)$. Thus, let $(\mathcal{U}_n:n\in\mathbb{N})$ be a sequence of elements of $\omeganbd$, say for each $n$ the set $V_n$ is a neighborhood of the identity element of $G$ and $\mathcal{V}_n = \{F\otimes V_n:F\subset G \mbox{ finite}\}$.

Then for each $n$ define $U_n = \bigcap\{V_j:j\le n\}$, a neighborhood of the identity of $G$, and define $\mathcal{U}_n = \{F\otimes U_n:n\in \mathbb{N}\}$. As each $\mathcal{U}_n$ is an element of $\omeganbd$, apply $\sfin(\omeganbd,\open^{wgp})$ to the sequence $(\mathcal{U}_n:n\in\mathbb{N})$. For each $n$ choose a finite subset $\mathcal{F}_n$ of $\mathcal{U}_n$ such that $\mathcal{G} = \bigcup\{\mathcal{F}_n:n\in\mathbb{N}\}$ is a weakly groupable cover of $G$. 

Fix a partition $(\mathcal{G}_n:n\in\mathbb{N})$ of $\mathcal{G}$ into finite sets $\mathcal{G}_n$ such that there is for each finite subset $S$ of $G$ an $n$ with $S\subset \bigcup\mathcal{G}_n$.
\end{proof}

Thus, Theorem \ref{thm:finpowersprod} establishes the consistency of the statement that if a $\textsf{T}_0$ group $(G,\otimes)$ is such that $(G,\otimes) \times (G,\otimes)$ has the property $\sfin(\onbd,\open)$, then every finite power of $(G,\otimes)$ has the property $\sfin(\onbd,\open)$. It turns out that in fact this statement is independent of \textsf{ZFC}, since on the other hand:
\begin{theorem}[\cite{MST}, Theorem 11]\label{thm:kk+1}
It is consistent, relative to the consistency of \textsf{ZFC}, that there is for each positive integer $k$ a separable metrizable topological group $(G,\otimes)$ such that $G^k$ has the property $\sfin(\onbd,\open)$ while $G^{k+1}$ does not have the property $\sfin(\onbd,\open)$.
\end{theorem}
 
 Less is known about version (A) of the question above. Interestingly, for the subclass of metrizable groups which have the property $\sfin(\onbd,\open)$ in all finite powers, it is consistent that a product of finitely many groups in this subclass is still in this subclass. In fact, an equiconsistency criterion has been identified.
\begin{theorem}[He, Tsaban and Zang \cite{HTZ}, Theorem 2.1]\label{thm:powertoprooducts}
The following statements are equivalent:
\begin{enumerate}
\item{\textsf{NCF}}
\item{The product of two metrizable groups, each with the property $\sfin(\onbd,\open^{wgp})$, is a topological group with the property $\sfin(\onbd,\open^{wgp})$.}
\end{enumerate}
\end{theorem}

This result raises the following potentially more modest analogue of the version (A) question:
\begin{problem}
Is it consistent that product of any two $\textsf{T}_0$  groups, each with the property $\sfin(\onbd,\open^{wgp})$, is a topological group with the property $\sfin(\onbd,\open^{wgp})$?
\end{problem}

\section{Further remarks on $\aleph_0$ bounded P groups}\label{sec:Pgroups}  

Towards further strengthening results about $\aleph_0$-bounded $P$ groups we next consider products of topological groups with the property $\sone(\onbd,\open)$, a stronger property than $\sfin(\onbd,\open)$. 
For finite powers there is the following prior result
\begin{theorem}[\cite{BKS}, Theorem 15] \label{thm:sfinpowers} For a topological group $(G,\otimes)$ the following are equivalent:
\begin{enumerate}
\item{Each finite power of $(G,\otimes)$ has the property $\sone(\onbd,\open)$}
\item{$(G,\otimes)$ has the property $\sone(\onbd,\open^{wgp})$}
\end{enumerate}
\end{theorem}

{\flushleft Lemma \ref{lemma:obPgp1} can be strengthened as follows:}
\begin{lemma}\label{lemma:obPgp3}
Any $\aleph_0$-bounded $P$ group has the property $\sone(\onbd,\open^{gp})$
\end{lemma}
\begin{proof}
Let $(G,\otimes)$ be an $\aleph_0$-bounded $P$ group.  Let $(\mathcal{U}_n:n\in\mathbb{N})$ be a sequence of $\onbd$-covers of $G$. For each $n$ choose a neighborhood $U_n$ of the identity such that $\mathcal{U}_n = \mathcal{O}(U_n) = \{g\otimes U_n:g\in G\}$. Since $G$ is a $P$ space the $\textsf{G}_{\delta}$ set $U = \bigcap\{U_n:n\in\mathbb{N}\}$ is an open neighborhood of the identity, and the $\onbd$-cover $\{g\otimes U:g\in G\}$ of $G$ is a refinement of each of the $\onbd$ covers $\mathcal{U}_n$. For each $n$ set $\mathcal{V}_n$ be the $\omeganbd$ cover $\{F\otimes U:\; F\subset G \mbox{ finite}\}$.  

As was shown in Theorem \ref{thm:PgpSel}, this group has the selection property $\sone(\omeganbd,\Gamma)$. Applying this selection property of $G$ to the sequence $(\mathcal{V}_n:n\in\mathbb{N})$ we find for each $n$ a set $V_n\in\mathcal{V}_n$ such that $\{V_n:n\in\mathbb{N}\}$ is a $\gamma$-cover of $G$ - that is, for each $g\in G$ we have for all but finitely many $n$ that $g\in V_n$. For each $n$ fix the finite set $F_n\subset G$ such that $V_n = F_n\otimes U$. Now let $n_1,\; n_2,\; \cdots,\; n_k,\; \cdots$ be natural numbers such that for each $i$ we have $n_i = \vert F_i\vert$. Next choose elements $g_1,\; g_2,\; \cdots,\; g_n,\; \cdots$ from $G$ as follows: $g_1,\cdots,g_{n_1}$ lists the distinct elements of $F_1$, $g_{n_1+1},\cdots,g_{n_1+n_2}$ lists the distinct elements of $F_2$,  and in general
$\{g_{n_1+\cdots+n_{k-1}+1},\cdots,g_{n_1+n_2+\cdots+n_k}\}$ lists the distinct elements of $F_k$, and so on.
Thus for each $k$ we have $V_k = \cup\{g_i\otimes U: n_1+\cdots+n_{k-1} < i \le n_1+\cdots+n_k\}$.

{\flushleft{\bf Claim: }} $\{g_i\otimes U_i:i\in\mathbb{N}\}$ is a groupable open cover of $G$.
For let  an $h\in G$ be given. Since $(V_n:n\in\mathbb{N})$ is a $\gamma$ cover of $G$, fix a $k$ such that for all $m\ge k$ it is true that $h\in V_k$. Then for all $m\ge k$, the element $g$ of $G$ is in $\bigcup\{g_i\otimes U_i:n_1+\cdots+n_{k-1}<i\le n_1+\cdots+n_k\}$, confirming that the selector $(g_i\otimes U_i:i\in\mathbb{N})$ of the original sequence of $\onbd$ covers is a groupable open cover of $G$. 
\end{proof}

Lemma \ref{lemma:obPgp3} provides the following alternative derivation that $\aleph_0$-bounded $P$ groups have the property $\sone(\omeganbd,\Gamma)$:
\begin{lemma}\label{lemma:gptogamma}
If a topological group has the property $\sone(\onbd,\open^{gp})$ then it has the property $\sone(\omeganbd,\Gamma)$.
\end{lemma}
\begin{proof}
Let $(G,\otimes)$ be a topological group which has the property $\sone(\onbd,\open^{gp})$. Let $(\mathcal{U}_n:n\in\mathbb{N})$ be a sequence of $\omeganbd$ covers of $G$. For each $n$ fix $U_n$, the neighborhood of the identity element for which $\mathcal{U}_n = \{F\otimes U_n:F\subset G \mbox{ finite}\}$.

For each $n$ set $V_n = \bigcap\{U_j:j\le n\}$, a neighborhood if the identity element of the group $(G,\otimes)$. Set $\mathcal{V}_n = \{F\otimes V_n:F\subset G \mbox{ finite}\}$, a member of $\omeganbd$ that refines $\mathcal{U}_n$. 

Now apply to selection principle $\sone(\onbd,\open^{gp})$ to each of the $\onbd$ covers $\mathcal{A}_n = \{g\otimes V_n: n\in\mathbb{N}\}$: For each $n$ choose  
an $A_n\in\mathcal{A}_n$ such that $\{A_n:n\in\mathbb{N}\} \in\open^{gp}$. Fix a sequence $n_1<n_2<\cdots<n_k<\cdots$ of natural numbers such that for each $g\in G$, for all but finitely many $k$, $g$ is an element of $\bigcup\{A_m:n_{k-1}<m\le n_k\}$. For each $m$ fix $g_m\in G$ such that $A_m = g_m\otimes V_m$. Then define finite sets $F_1 = \{g_n:n \le n_1\}$ and for each $k$, $F_k = \{g_j: n_{k-1}<j \le n_k\}$.

For each $k$ set $U_k = F_k\otimes V_k$, an element of $\mathcal{U}_k$. Then $\{U_k:k\in\mathbb{N}\}$ is a $\gamma$-cover of $G$, 
for let an $x\in G$ be given. Choose $k$ so large that for all $m\ge k$, $x$ is a member of $\bigcup_{n_{m-1}<j\le n_m}g_j\otimes V_j$. Since $\bigcup_{n_{m-1}<j\le n_m}g_j\otimes V_j \subseteq U_m$, it follows that for all $m\ge k$, $x$ is a member of $U_m$. 
\end{proof}

\section{Conclusion}

In this paper we merely touched on four extensively explored topics in the arena of $\aleph_0$-bounded groups. Among the numerous exploration possibilities we pose here only the following one about cardinalities: 

We have noted that there are no \emph{a priori} theoretical restrictions on the cardinality that a $T_0$ group with the property $\sfin(\onbd,\open)$, or even $\sone(\omeganbd,\Gamma)$, can have. Each infinite cardinality is possible. As was pointed out in Theorems 8 and Corollary 17 of \cite{MS1}, the same holds for $T_0$ groups with the property $\sone(\onbd,\open)$, or even the much stronger property that the group is an $\aleph_0$-bounded P group, or a group in which TWO has a winning strategy in the game $\gone(\onbd,\open)$. However, the following issue regarding the achievable cardinality for a given type of $\aleph_0$-bounded group is much more subtle\footnote{Also, in the class of Lindel\"of spaces, for example, there are no constraints on the cardinalities achievable in the class oof $T_0$ Lindel\"of spaces, yet there are constraints on the cardinalities of subspaces that are Lindel\"of, as can for example be gleaned from \cite{KT}. 
}: Let an $\aleph_0$-bounded $T_0$ group be given. It necessarily has subgroups with properties $\sfin(\onbd,\open)$, $\sone(\omeganbd,\Gamma)$, $\sone(\onbd,\open)$, and any of the other nonempty selection based classes obtained by varying the types of open covers appearing. The question of what cardinality restrictions there may be on subgroups of an $\aleph_0$ bounded $T_0$ group has been extensively studied in the case of the property $\sone(\onbd,\open)$. For example in \cite{GS} and in \cite{MS2} the following hypothesis\footnote{This hypothesis is a generalization of the classical  Borel Conjecture} is investigated:
\begin{quote}
Each subgroup with property $\sone(\onbd,\open)$ of an $\aleph_0$-bounded group of weight $\kappa$ has cardinality at most $\kappa$ 
\end{quote}
It would be interesting to know if there are similar feasible hypotheses of cardinality bounds for subgroups with property $\sone(\omeganbd,\Gamma)$ or with property $\sfin(\onbd,\open)$ of $\aleph_0$ bounded groups that are not $\sigma$ totally bounded.

\end{document}